\documentclass[12pt,a4paper]{amsart}
\usepackage{amsmath}
\usepackage{amsthm}
\usepackage{amssymb}
\usepackage{tikz-cd}
\usepackage[english]{babel}
\usepackage{verbatim}
\usepackage[shortlabels]{enumitem}
\usepackage{mathrsfs}

\newtheoremstyle{mio}%
	{}{} 
	{\itshape}{} 
	{\bfseries}{.}{ } 
	{#1 #2\thmnote{\mdseries~(\scshape #3)}} 
\theoremstyle{mio}
\newtheorem{teor}{Theorem}[section]
\newtheorem{cor}[teor]{Corollary}
\newtheorem{prop}[teor]{Proposition}
\newtheorem{lemma}[teor]{Lemma}

\newtheoremstyle{definition2}%
	{}{} 
	{}{} 
	{\bfseries}{.}{ } 
	{#1 #2\thmnote{\mdseries~ #3}} 
\theoremstyle{definition2}

\newcommand{\Spec}{\mathrm{Spec}}
\DeclareMathOperator{\Max}{Max}
\newcommand{\Zar}{\mathrm{Zar}}
\newcommand{\Over}{\mathrm{Over}}
\newcommand{\comp}{\mathrm{comp}}
\newcommand{\B}{\mathcal{B}}
\newcommand{\Kr}{\mathrm{Kr}}
\newcommand{\Zarmin}{\Zar_{\mathrm{min}}}
\DeclareMathOperator{\trdeg}{trdeg}
\newcommand{\insQ}{\mathbb{Q}}

\title{When the Zariski space is a Noetherian space}
\author{Dario Spirito}
\date{\today}
\address{Dipartimento di Matematica e Fisica, Universit\`a degli Studi ``Roma Tre'', Roma, Italy}
\email{spirito@mat.uniroma3.it}

\begin{document}
\begin{abstract}
We characterize when the Zariski space $\Zar(K|D)$ (where $D$ is an integral domain, $K$ is a field containing $D$ and $D$ is integrally closed in $K$) and the set $\Zarmin(L|D)$ of its minimal elements are Noetherian spaces.
\end{abstract}

\maketitle

\section{Introduction}
The \emph{Zariski space} $\Zar(K|D)$ of the valuation ring of a field $K$ containing a subring $D$ was introduced by O. Zariski (under the name \emph{abstract Riemann surface}) during its study of resolution of singularities \cite{zariski_sing,zariski_comp}. In particular, he introduced a topology on $\Zar(K|D)$ (which was later called \emph{Zariski topology}) and proved that it makes $\Zar(K|D)$ into a compact space \cite[Chapter VI, Theorem 40]{zariski_samuel_II}. Later, the Zariski topology on $\Zar(K|D)$ was studied more carefully, showing that it is a \emph{spectral space} in the sense of Hochster \cite{hochster_spectral}, i.e., that there is a ring $R$ such that the spectrum of $R$ (endowed with the Zariski topology) is homeomorphic to $\Zar(K|D)$ \cite{dobbs_fedder_fontana,fontana_krr-abRs,fifolo_transactions}. This topology has also been used to study representations of an integral domain by intersection of valuation rings \cite{olberding_noetherianspaces,olberding_affineschemes,olberding_topasp} and, for example, in real and rigid algebraic geometry \cite{hub-kneb,schwartz-compactification}.

In \cite{ZarNoeth}, it was shown that in many cases $\Zar(D)$ is not a Noetherian space, i.e., there are subspaces of $\Zar(D)$ that are not compact. In particular, it was shown that $\Zar(D)\setminus\{V\}$ (where $V$ is a minimal valuation overring of $D$) is often non-compact: for example, this happens when $\dim(V)>2\dim(D)$ \cite[Proposition 4.3]{ZarNoeth} or when $D$ is Noetherian and $\dim(V)\geq 2$ \cite[Corollary 5.2]{ZarNoeth}.

In this paper, we study integral domains such that $\Zar(D)$ is a Noetherian space, and, more generally, we study when the Zariski space $\Zar(K|D)$ is Noetherian. We show that, if $D=F$ is a field, then $\Zar(K|F)$ can be Noetherian only if the transcendence degree of $K$ over $F$ is at most 1 and, when $\trdeg_FK=1$, we characterize when this happens in terms of the extensions of the valuation domains of $F[X]$, where $X$ is an element of $K$ transcendental over $F$ (Proposition \ref{prop:trdeg1}). In Section \ref{sect:domain}, we study the case where $K$ is the quotient field of $D$: we first consider the local case, showing that if $\Zar(D)$ is Noetherian then $D$ must be a pseudo-valuation domain (Theorem \ref{teor:pvd}) and, subsequently, we globalize this result to the non-local case, showing that $\Zar(D)$ is Noetherian if and only if so are $\Spec(D)$ and $\Zar(D_M)$, for every maximal ideal $M$ of $D$ (Theorem \ref{teor:global} and Corollary \ref{cor:global}). We also prove the analogous results for the set $\Zarmin(K|D)$ of the minimal elements of $\Zar(K|D)$.

\section{Background}
\subsection{Overrings and the Zariski space}
Let $D$ be an integral domain and let $K$ be a ring containing $D$. We define $\Over(K|D)$ as the set of rings contained between $D$ and $K$. The \emph{Zariski topology} on $\Over(K|D)$ is the topology having, as a subbasis of closed sets, the sets in the form
\begin{equation*}
\B(x_1,\ldots,x_n):=\{V\in\Over(K|D)\mid x_1,\ldots,x_n\in V\},
\end{equation*}
as $x_1,\ldots,x_n$ range in $K$. If $K$ is the quotient field of $D$, an element of $\Over(K|D)$ is called an \emph{overring} of $D$.

If $K$ is the quotient field of $D$, a subset $X\subseteq\Over(K|D)$ is a \emph{locally finite family} if every $x\in D$ (or, equivalently, every $x\in K$) is a non-unit in only finitely many $T\in\Over(K|D)$.

If $K$ is a field containing $D$, the \emph{Zariski space} of $D$ in $K$ is the set of all valuation domains containing $D$ and whose quotient field is $K$; we denote it by $\Zar(K|D)$. The Zariski topology on $\Zar(K|D)$ is simply the Zariski topology inherited from $\Over(K|D)$. If $K$ is the quotient field of $D$, then $\Zar(K|D)$ will simply be denoted by $\Zar(D)$, and its elements are called the \emph{valuation overrings} of $D$.

Under the Zariski topology, $\Zar(K|D)$ is compact \cite[Chapter VI, Theorem 40]{zariski_samuel_II}.

We denote by $\Zarmin(K|D)$ the set of minimal elements of $\Zar(K|D)$, with respect to containment. If $V$ is a valuation domain, we denote by $\mathfrak{m}_V$ its maximal ideal. Given $X\subseteq\Zar(D)$, we define
\begin{equation*}
X^\uparrow:=\{V\in\Zar(D)\mid V\supseteq W\text{~for some~}W\in X\}.
\end{equation*}
Since a family of open sets is a cover of $X$ if and only if it is a cover of $X^\uparrow$, we have that $X$ is compact if and only if $X^\uparrow$ is compact.

If $X$ is a subset of $\Zar(D)$, we denote by $A(X)$ the intersection $\bigcap\{V\mid V\in X\}$, called the \emph{holomorphy ring} of $X$ \cite{roquette-holomorphy}. Clearly, $A(X)=A(X^\uparrow)$.

The \emph{center map} is the application
\begin{equation*}
\begin{aligned}
\gamma\colon\Zar(K|D) & \longrightarrow \Spec(D)\\
V & \longmapsto \mathfrak{m}_V\cap D.
\end{aligned}
\end{equation*}
If $\Zar(K|D)$ and $\Spec(D)$ are endowed with the respective Zariski topologies, the map $\gamma$ is continuous (\cite[Chapter VI, \textsection 17, Lemma 1]{zariski_samuel_II} or \cite[Lemma 2.1]{dobbs_fedder_fontana}), surjective (this follows, for example, from \cite[Theorem 5.21]{atiyah} or \cite[Theorem 19.6]{gilmer}) and closed \cite[Theorem 2.5]{dobbs_fedder_fontana}.

In studying $\Zar(K|D)$, it is usually enough to consider the case where $D$ is integrally closed in $K$; indeed, if $\overline{D}$ is the integral closure of $D$ in $K$, then $\Zar(K|D)=\Zar(K|\overline{D})$.

\subsection{Noetherian spaces}
A topological space $X$ is \emph{Noetherian} if its open sets satisfy the ascending chain condition, or equivalently if all its subsets are compact. If $X=\Spec(R)$ is the spectrum of a ring, then $X$ is a Noetherian space if and only if $R$ satisfies the ascending chain condition on radical ideals; in particular, the spectrum of a Noetherian ring is always a Noetherian space. If $\Spec(R)$ is Noetherian, then every ideal of $R$ has only finitely many minimal primes (see e.g. the proof of \cite[Chapter 4, Corollary 3, p.102]{bourbaki_ac} or \cite[Chapter 6, Exercises 5 and 7]{atiyah}).

Every subspace and every continuous image of a Noetherian space is again Noetherian; in particular, if $\Zar(D)$ is Noetherian then so are $\Zarmin(D)$ and $\Spec(D)$ \cite[Proposition 4.1]{ZarNoeth}.

\subsection{Kronecker function rings}
Let $K$ be the quotient field of $D$. For every $V\in\Zar(D)$, let $V^b:=V[X]_{\mathfrak{m}_V[X]}\subseteq K(X)$. If $\Delta\subseteq\Zar(D)$, the \emph{Kronecker function ring} of $D$ with respect to $\Delta$ is
\begin{equation*}
\Kr(D,\Delta):=\bigcap\{V^b\mid V\in\Delta\};
\end{equation*}
we denote $\Kr(D,\Zar(D))$ simply by $\Kr(D)$.

The ring $\Kr(D,\Delta)$ is always a B\'ezout domain whose quotient field is $K(X)$, and, if $\Delta$ is compact, the intersection map $W\mapsto W\cap K$ establishes a homeomorphism between $\Zar(\Kr(D,\Delta))$ and the set $\Delta^\uparrow$ \cite{dobbs_fedder_fontana,fontana_krr-abRs,fifolo_transactions}. Since $\Kr(D,\Delta)$ is a Pr\"ufer domain, furthermore, $\Zar(\Kr(D,\Delta))$ is homeomorphic to $\Spec(\Kr(D,\Delta))$; hence, $\Spec(\Kr(D,\Delta))$ is homeomorphic to $\Delta^\uparrow$, and asking if $\Zar(D)$ is Noetherian is equivalent to asking if $\Spec(\Kr(D))$ is Noetherian or, equivalently, if $\Kr(D)$ satisfies the ascending chain condition on radical ideals.

See \cite[Chapter 32]{gilmer} or \cite{fontana_loper} for general properties of Kronecker function rings.

\subsection{Pseudo-valuation domains}
Let $D$ be an integral domain with quotient field $K$. Then, $D$ is called a \emph{pseudo-valuation domain} (for short, \emph{PVD}) if, for every prime ideal $P$ of $D$, whenever $xy\in P$ for some $x,y\in K$, then at least one between $x$ and $y$ is in $P$. Equivalently, $D$ is a pseudo-valuation domain if and only if it is local and its maximal ideal $M$ is also the maximal ideal of some valuation overring $V$ of $D$ (called the valuation domain \emph{associated} to $D$) \cite[Corollary 1.3 and Theorem 2.7]{pvd}. If $D$ is a valuation domain, then it is also a PVD, and the associated valuation ring is $D$ itself.

The prototypical examples of a pseudo-valuation domain that is not a valuation domain is the ring $F+XL[[X]]$, where $F\subseteq L$ is a field extension; its associated valuation domain is $L[[X]]$.

\section{Examples and reduction}
The easiest case for the study of the topology of $\Zar(D)$ is when $D$ is a Pr\"ufer domain, i.e., when $D_M$ is a valuation domain for every maximal ideal $M$ of $D$. 
\begin{prop}
Let $D$ be a Pr\"ufer domain. Then:
\begin{enumerate}[(a)]
\item $\Zar(D)$ is a Noetherian space if and only if $\Spec(D)$ is Noetherian;
\item $\Zarmin(D)$ is Noetherian if and only if $\Max(D)$ is Noetherian.
\end{enumerate}
\end{prop}
\begin{proof}
Since $D$ is Pr\"ufer, the center map $\gamma:\Zar(D)\longrightarrow\Spec(D)$ is a homeomorphism \cite[Proposition 2.2]{dobbs_fedder_fontana}. This proves the first claim; the second one follows from the fact that the minimal valuation overrings of $D$ correspond to the maximal ideals.
\end{proof}

Another example of a domain that has a Noetherian Zariski space is the pseudo-valuation domain $D:=\insQ+Y\insQ(X)[[Y]]$, where $X,Y$ are indeterminates on $\insQ$, since in this case $\Zar(D)$ can be written as the union of the quotient field of $D$ and two sets homeomorphic to $\Zar(\insQ[X])\simeq\Spec(\insQ[X])$, which are Noetherian; from this, it is possible to build examples of non-Pr\"ufer domain whose Zariski spectrum is Noetherian, and having arbitrary finite dimension \cite[Example 4.7]{ZarNoeth}.

More generally, we have the following result, which is probably well-known.
\begin{lemma}\label{lemma:pullback-Zar}
Let $D$ be an integral domain, and suppose that a prime ideal $P$ of $D$ is also the maximal ideal of a valuation overring $V$ of $D$. Then, the quotient map $\pi:V\longrightarrow V/P$ establishes a homeomorphism between $\{W\in\Zar(D)\mid W\subseteq V\}$ and $\Zar(V/P|D/P)$, and between $\Zarmin(D)$ and $\Zarmin(V/P|D/P)$.
\end{lemma}
\begin{proof}
Consider the set $\Over(V|D)$ and $\Over(V/P|D/P)$. Then, the map
\begin{equation*}
\begin{aligned}
\widetilde{\pi}\colon\Over(V|D) & \longrightarrow\Over(V/P|D/P)\\
A & \longmapsto \pi(A)=A/P
\end{aligned}
\end{equation*}
is a bijection, whose inverse is the map sending $B$ to $\pi^{-1}(B)$. Furthermore, it is a homeomorphism: indeed, if $x\in V/P$ then $\widetilde{\pi}^{-1}(\B(x))=\B(y)$, for any $y\in\pi^{-1}(x)$, while if $x\in V$ then $\widetilde{\pi}(\B(x))=\B(\pi(x))$.

The condition on $P$ implies that $D$ is the pullback of the diagram
\begin{equation*}
\begin{tikzcd}
D\arrow[hookrightarrow]{d}\arrow[two heads]{r}{\pi} & D/P\arrow[hookrightarrow]{d}\\
V\arrow[two heads]{r}{\pi} & V/P;
\end{tikzcd}
\end{equation*}
hence, every $A\in\Over(V|D)$ arises as a pullback. By \cite[Theorem 2.4(1)]{topologically-defined}, $A$ is a valuation domain if and only if $\pi(A)$ is a valuation domain and $V/P$ is the quotient field of $\pi(A)$; hence, $\widetilde{\pi}$ restricts to a bijection between $\Zar(D)\cap\Over(V|D)=\{W\in\Zar(D)\mid W\subseteq V\}$ and $\Zar(V/P|D/P)$. Furthermore, since $\widetilde{\pi}$ is a homeomorphism, so is its restriction. The claim about $\Zar(D)$ and $\Zar(V/P|D/P)$ is proved; the claim for the space of minimal elements follows immediately.
\end{proof}

\begin{prop}\label{prop:riduzione}
Let $D$ be an integral domain, and let $L$ be a field containing $D$. Then, there is a ring $R$ such that:
\begin{itemize}
\item $\Zar(L|D)\simeq\Zar(R)\setminus\{F\}$, where $F$ is the quotient field of $R$;
\item $\Zarmin(L|D)\simeq\Zarmin(R)$.
\end{itemize} 
\end{prop}
\begin{proof}
Let $X$ be an indeterminate over $L$, and define $R:=D+XL[[X]]$. Then, the prime ideal $P:=XL[[X]]$ of $R$ is also a prime ideal of the valuation domain $L[[X]]$; by Lemma \ref{lemma:pullback-Zar}, it follows that $\Zar(L|D)\simeq\Delta:=\{W\in\Zar(R)\mid W\subseteq L[[X]]\}$. Furthermore, every valuation overring $V$ of $R$ contains $XL[[X]]$, and thus it is either in $\Delta$ or properly contains $L[[X]]$; however, since $L[[X]]$ has dimension 1, the latter case is possible only if $V=L((X))$ is the quotient field of $R$. The first claim is proved, and the second follows easily.
\end{proof}

While Proposition \ref{prop:riduzione} shows that (theoretically) we only need to consider spaces of valuation overrings, it is usually easier to not be restricted to this case; the following Proposition \ref{prop:noncampo} is an example, as will be the analysis of field extensions done in Section \ref{sect:field}.

\begin{prop}\label{prop:noncampo}
Let $D$ be an integral domain that is not a field, let $K$ be its quotient field and $L$ a field extension of $K$. If $\trdeg_KL\geq 1$, then $\Zar(L|D)$ and $\Zarmin(L|D)$ are not Noetherian.
\end{prop}
\begin{proof}
If $\trdeg_KL\geq 1$, there is an element $X\in L\setminus K$ that is not algebraic over $L$. If $\Zar(L|D)$ is Noetherian, so is its subset $\Zar(L|D[X])$, and thus also $\Zar(K(X)|D[X])=\Zar(D[X])$, which is the (continuous) image of $\Zar(L|D[X])$ under the intersection map $W\mapsto W\cap K(X)$. However, since $D$ is not a field, $\Zar(D[X])$ is not Noetherian by \cite[Proposition 5.4]{ZarNoeth}; hence, neither $\Zar(L|D)$ can be Noetherian.

Consider now $\Zarmin(L|D)$: it projects onto $\Zarmin(K(X)|D)$, and thus we can suppose that $L=K(X)$. Let $V$ be a minimal valuation overring of $D$: then, there is an extension $W$ of $V$ to $L$ such that $X$ is the generator of the maximal ideal of $W$; furthermore, $W$ belongs to $\Zarmin(K(X)|D)$. In particular, $\Spec(W)\setminus\Max(W)$ has a maximum, say $P$. Let $\Delta:=\Zar(L|D)\setminus\{W\}$: then, $\Delta$ can be written as the union of $\Lambda:=(\Zarmin(L|D)\setminus\{W\})^\uparrow$ and $\{W_P\}^\uparrow$. The latter is compact since $\{W_P\}$ is compact; if $\Zarmin(L|D)\setminus\{W\}$ were compact, so would be $\Lambda$. In this case, also $\Delta$ would be compact, against the proof of \cite[Proposition 5.4]{ZarNoeth}. Hence, $\Delta$ is not compact, and so $\Zarmin(L|D)$ is not Noetherian.
\end{proof}

\section{Field extensions}\label{sect:field}
In this section, we consider a field extension $F\subseteq L$ and analyze when the Zariski space $\Zar(L|F)$ and its subset $\Zarmin(L|F)$ are Noetherian. By Proposition \ref{prop:riduzione}, this is equivalent to studying the Zariski space of the pseudo-valuation domain $F+XL[[X]]$.

This problem naturally splits into three cases, according to whether the transcendence degree of $L$ over $F$ is $0$, $1$ or at least $2$. The first and the last cases have definite answers, and we collect them in the following proposition. Part \ref{prop:trdeg02:2} is a slight generalization of \cite[Corollary 5.5(b)]{ZarNoeth}.
\begin{prop}\label{prop:trdeg02}
Let $F\subseteq L$ be a field extension.
\begin{enumerate}[(a)]
\item\label{prop:trdeg02:0} If $\trdeg_FL=0$, then $\Zar(L|F)=\{L\}=\Zarmin(L|D)$, and in particular they are Noetherian.
\item\label{prop:trdeg02:2} If $\trdeg_FL\geq 2$, then $\Zar(L|F)$ and $\Zarmin(L|F)$ are not Noetherian.
\end{enumerate}
\end{prop}
\begin{proof}
\ref{prop:trdeg02:0} is obvious. For \ref{prop:trdeg02:2}, let $X,Y$ be elements of $L$ that are algebraically independent. Then, the intersection map $\Zarmin(L|F)\longrightarrow\Zarmin(F(X,Y)|F)$ is surjective, and thus it is enough to prove that $\Zarmin(F(X,Y)|F)$ is not Noetherian.

Let $V\in\Zarmin(F(X,Y)|F)$ and, without loss of generality, suppose $X,Y\in V$. Let $\Delta:=\Zarmin(F(X,Y)|F)\setminus\{V\}$. Then, $\Lambda:=\Zar(F(X,Y)|F)\setminus\{V\}$ is the union of $\Delta^\uparrow$ and a finite set (the valuation domains properly containing $V$). If $\Delta$ were compact, so would be $\Lambda$; hence, so would be $\Lambda\cap\Zar(F[X,Y])$ (since both $\Lambda$ and $\Zar(F[X,Y])$ would be closed in the inverse topology; see e.g. \cite[Remark 2.2 and Proposition 2.6]{fifolo_transactions}). However, $\Lambda\cap\Zar(F[X,Y])=\Zar(F[X,Y])\setminus\{V\}$, which is not compact by the proof of \cite[Proposition 5.4]{ZarNoeth}. Hence, $\Lambda$ is not compact, and thus $\Delta$ cannot be compact. Hence, $\Zarmin(F(X,Y)|F)$ is not Noetherian.
\end{proof}

On the other hand, the case of transcendence degree 1 is more subtle.  In \cite[Corollary 5.5(a)]{ZarNoeth}, it was showed that $\Zar(L|F)$ is Noetherian if $L$ is finitely generated over $F$; we now state a characterization.
\begin{prop}\label{prop:trdeg1}
Let $F\subseteq L$ be a field extension such that $\trdeg_FL=1$. Then, the following are equivalent:
\begin{enumerate}[(i)]
\item\label{prop:trdeg1:LF} $\Zar(L|F)$ is Noetherian;
\item\label{prop:trdeg1:LFmin} $\Zarmin(L|F)$ is Noetherian;
\item\label{prop:trdeg1:every-fm} for every $X\in L$ transcendental over $F$, every valuation on $F[X]$ has only finitely many extensions to $L$;
\item\label{prop:trdeg1:ext-fm} there is an $X\in L$, transcendental over $F$, such that every valuation on $F[X]$ has only finitely many extensions to $L$;
\item\label{prop:trdeg1:every-ic} for every $X\in L$ transcendental over $F$, the integral closure of $F[X]$ in $L$ has Noetherian spectrum;
\item\label{prop:trdeg1:ext-ic} there is an $X\in L$, transcendental over $F$, such that the integral closure of $F[X]$ in $L$ has Noetherian spectrum.
\end{enumerate}
\end{prop}
\begin{proof}
Every valuation domain of $L$ containing $F$ must contain the algebraic closure of $F$ is $L$; hence, without loss of generality we can suppose that $F$ is algebraically closed in $L$.

\ref{prop:trdeg1:LF} $\Longrightarrow$ \ref{prop:trdeg1:LFmin} is obvious; \ref{prop:trdeg1:LFmin} $\Longrightarrow$ \ref{prop:trdeg1:LF} follows since (being $\trdeg_FL=1$) $\Zar(L|F)=\Zarmin(L|F)\cup\{L\}$.

\ref{prop:trdeg1:LF} $\Longrightarrow$ \ref{prop:trdeg1:every-fm}. Take $X\in L\setminus F$, and suppose there is a valuation $w$ on $F[X]$ with infinitely many extensions to $L$; let $W$ be the valuation domain corresponding to $w$. Then, the integral closure $\overline{W}$ of $W$ in $L$ would have infinitely many maximal ideals. Since every maximal ideal of $\overline{W}$ contains the maximal ideal of $W$, the Jacobson radical $J$ of $\overline{W}$ contains the maximal ideal of $W$, and in particular it is nonzero. It follows that $J$ has infinitely many minimal primes; hence, $\Max(\overline{W})$ is not a Noetherian space. However, $\Max(\overline{W})$ is homeomorphic to a subspace of $\Zar(L|F)$, which is Noetherian by hypothesis; this is a contradiction, and so every valuation has only finitely many extensions.

\ref{prop:trdeg1:every-fm} $\Longrightarrow$ \ref{prop:trdeg1:every-ic}. Let $T$ be the integral closure of $F[X]$. If $\Spec(T)$ is not Noetherian, then $T$ is not locally finite; i.e., there is an $\alpha\in T$ such that there are infinitely many maximal ideals of $T$ containing $\alpha$. Consider the norm $N(\alpha)$ of $\alpha$ over $F[X]$, i.e., the product of the algebraic conjugates of $\alpha$ over $F[X]$. Then, $N(\alpha)\neq 0$, and it is both an element of $F[X]$ (being equal to the constant term of the minimal polynomial of $F[X]$ over $\alpha$) and an element of every maximal ideal containing $\alpha$ (since all the conjugates are in $T$). Since every maximal ideal of $F[X]$ is contained in only finitely many maximal ideals of $T$ (since a maximal ideal of $F[X]$ correspond to a valuation $v$ and the maximal ideals of $T$ containing it to the extensions of $v$), it follows that $N(\alpha)$ is contained in infinitely many maximal ideals of $F[X]$. However, this contradicts the Noetherianity of $\Spec(F[X])$; hence, $\Spec(T)$ is Noetherian.

Now  \ref{prop:trdeg1:every-fm} $\Longrightarrow$ \ref{prop:trdeg1:ext-fm} and \ref{prop:trdeg1:every-ic} $\Longrightarrow$ \ref{prop:trdeg1:ext-ic} are obvious, while the proof of \ref{prop:trdeg1:ext-fm} $\Longrightarrow$ \ref{prop:trdeg1:ext-ic} is exactly the same as the previous paragraph; hence, we need only to show \ref{prop:trdeg1:ext-ic} $\Longrightarrow$ \ref{prop:trdeg1:LF}; the proof is similar to the one of \cite[Corollary 5.5(a)]{ZarNoeth}.

Let $X\in L$, $X$ transcendental over $F$, be such that the spectrum of the integral closure $T$ of $F[X]$ is Noetherian. Since $X$ is transcendental over $F$, there is an $F$-isomorphism $\phi$ of $F(X)$ sending $X$ to $X^{-1}$; moreover, we can extend $\phi$ to an $F$-isomorphism $\overline{\phi}$ of $L$. Since $\phi(F[X])=F[X^{-1}]$, the integral closure $T$ of $F[X]$ is sent by $\overline{\phi}$ to the integral closure $T'$ of $F[X^{-1}]$; in particular, $T\simeq T'$, and $\Spec(T)\simeq\Spec(T')$. Thus, also $\Spec(T')$ is Noetherian, and so is $\Spec(T)\cup\Spec(T')$. Furthermore, $\Zar(T)\simeq\Spec(T)\simeq\Spec(L|F[X])$, and analogously for $T'$; hence, $\Zar(T)\cup\Zar(T')$ is Noetherian. But every $W\in\Zar(L|F)$ contains at least one between $X$ and $X^{-1}$, and thus $W$ contains $F[X]$ or $F[X^{-1}]$; i.e., $W\in\Zar(T)$ or $W\in\Zar(T')$. Hence, $\Zar(L|F)=\Zar(T)\cup\Zar(T')$ is Noetherian.
\end{proof}

We remark that there are field extensions that satisfy the conditions of Proposition \ref{prop:trdeg1} without being finitely generated. For example, if $L$ is purely inseparable over some $F(X)$, then every valuation on $F[X]$ extends uniquely to $L$, and thus condition \ref{prop:trdeg1:every-fm} of the previous proposition is fulfilled; more generally, each valuation on $F(X)$ extends in only finitely many ways when the separable degree $[L:F(X)]_s$ is finite \cite[Corollary 20.3]{gilmer}. There are also examples in characteristic 0: for example, \cite[Section 12.2]{ribenboim} gives examples of non-finitely generated algebraic extension $F$ of the rational numbers such that every valuation on $\insQ$ has only finitely many extensions to $F$. The same construction works also on $\insQ(X)$, and if $L$ is such an example then $\insQ\subseteq L$ will satisfy the conditions of Proposition \ref{prop:trdeg1}.

\section{The domain case}\label{sect:domain}
We now want to study when the space $\Zar(D)$ is Noetherian, where $D$ is an integral domain; without loss of generality, we can suppose that $D$ is integrally closed, since $\Zar(D)=\Zar(\overline{D})$. We start by studying intersections of Noetherian families of valuation rings.

Recall that a \emph{treed domain} is an integral domain whose spectrum is a tree (i.e., such that, if $P$ and $Q$ are non-comparable prime ideals, then they are coprime). In particular, every Pr\"ufer domain is treed.
\begin{lemma}\label{lemma:treed-MaxR}
Let $R$ be a treed domain. If $\Max(R)$ is Noetherian, then every ideal of $R$ has only finitely many minimal primes.
\end{lemma}
Note that we cannot improve this result to $\Spec(R)$ being Noetherian: for example, the spectrum of a valuation domain with unbranched maximal ideal if not Noetherian, while its maximal spectrum -- a singleton -- is Noetherian.
\begin{proof}
Let $I$ be an ideal of $R$, and let $\{P_\alpha\mid\alpha\in A\}$ be the set of its minimal prime ideals. For every $\alpha$, choose a maximal ideal $M_\alpha$ containing $P_\alpha$; note that $M_\alpha\neq M_\beta$ if $\alpha\neq\beta$, since $R$ is treed. Let $\Lambda$ be the set of the $M_\alpha$.

Let $X\subseteq\Lambda$, and define $J(X):=\bigcap\{IR_M\mid M\in X\}\cap R$: we claim that, if $M\in\Lambda$, then $J(X)\subseteq M$ if and only if $M\in X$.  Indeed, clearly $J(X)$ is contained in every element of $X$. On the other hand, suppose $N\in\Lambda\setminus X$. Since $\Max(R)$ is Noetherian, $X$ is compact, and thus also $\{R_M\mid M\in X\}$ is compact; by \cite[Corollary 5]{compact-intersections},
\begin{equation*}
J(X)R_N=\left(\bigcap_{M\in X}IR_M\right)R_N\cap R_N=\bigcap_{M\in X}IR_MR_N\cap R_N
\end{equation*}
Since $M,N\in\Lambda$, no prime contained in both $M$ and $N$ contains $I$; hence, $IR_MR_N$ contains 1 for each $M\in X$. Therefore, $1\in J(X)R_N$, i.e., $J(X)\nsubseteq N$.

Hence, every subset $X$ of $\Lambda$ is closed in $\Lambda$, since it is equal to the intersection between $\Lambda$ and the closed set of $\Spec(R)$ determined by $J(X)$. Since $\Lambda$ is Noetherian, it follows that $\Lambda$ must be finite; hence, also the set of minimal primes of $I$ is finite. The claim is proved.
\end{proof}

\begin{lemma}\label{lemma:VbWb}
Let $D$ be an integral domain with quotient field $K$, and let $V,W\in\Zar(D)$. If $VW=K$, then $V^bW^b=K(X)$.
\end{lemma}
\begin{proof}
Let $Z:=V^bW^b$. Then, Since $\Zar(D)$ and $\Zar(\Kr(D))$ are homeomorphic, $Z=(Z\cap K)^b$; however, $K\subseteq VW\subseteq V^bW^b$, and thus $Z\cap K=K$. It follows that $Z=K^b=K(X)$, as claimed.
\end{proof}

A consequence of Lemma \ref{lemma:treed-MaxR} is the following generalization of \cite[Theorem 3.4(2)]{olberding_noetherianspaces}.
\begin{teor}\label{teor:noeth-locfin}
Let $\Delta\subseteq\Zar(D)$ be a Noetherian space, and suppose that $VW=K$ for every $V\neq W$ in $\Delta$. Then, $\Delta$ is a locally finite space.
\end{teor}
\begin{proof}
Let $\Delta^b:=\{V^b\mid V\in\Delta\}$, and let $R:=\Kr(D,\Delta)$: then (since, in particular, $\Delta$ is compact), $\Zar(R)$ is equal to $(\Delta^b)^\uparrow$.

Since $R$ is a B\'ezout domain, it follows that $\Spec(R)\simeq(\Delta^b)^\uparrow$, while $\Max(R)\simeq\Delta^b$; in particular, $\Max(R)$ is Noetherian, and thus by Lemma \ref{lemma:treed-MaxR} every ideal of $R$ has only finitely many minimal primes. However, since $V^bW^b=K(X)$ for every $V\neq W$ in $\Delta$ (by Lemma \ref{lemma:VbWb}), it follows that every nonzero prime of $R$ is contained in only one maximal ideal; therefore, every ideal of $R$ is contained in only finitely many maximal ideals, and thus the family $\{R_M\mid M\in\Max(R)\}$ is locally finite. This family coincides with $\Delta^b$; since $\Delta^b$ is locally finite, also $\Delta$ must be locally finite, as claimed.
\end{proof}

We say that two valuation domains $V,W\in\Zar(D)\setminus\{K\}$ are \emph{dependent} if $VW\neq K$. Since $\Zar(D)$ is a tree, being dependent is an equivalence relation on $\Zar(D)\setminus\{K\}$; we call an equivalence class a \emph{dependency class}. If $\Zar(D)$ is finite-dimensional (i.e., if every valuation overring of $D$ has finite dimension) then the dependency classes of $\Zar(D)$ are exactly the sets in the form $\{W\in\Zar(D)\mid W\subseteq V\}$, as $V$ ranges among the one-dimensional valuation overrings of $D$.

Under this terminology, the previous theorem implies that, if $D$ is local and $\Zar(D)$ is Noetherian, then $\Zar(D)$ can only have finitely many dependency classes: indeed, otherwise, we could form a Noetherian but not locally finite subset of $\Zar(D)$ by taking one minimal overring in each dependency class, against the theorem. We actually can say (and will need) something more.

Given a set $X\subseteq\Zar(D)$, we define $\comp(X)$ as the set of all valuation overrings of $D$ that are comparable with some elements of $X$; i.e.,
\begin{equation*}
\comp(X):=\{W\in\Zar(D)\mid \exists~V\in X\text{~such that~}W\subseteq V\text{~or~}V\subseteq W\}.
\end{equation*}
If $X=\{V\}$ is a singleton, we write $\comp(V)$ for $\comp(X)$. Note that, for every subset $X$, $\comp(\comp(X))=\Zar(D)$, since $\comp(X)$ contains the quotient field of $D$.

The purpose of the following propositions is to show that, if $D$ is local and $\Zar(D)$ is Noetherian, then $\Zar(D)$ can be written as $\comp(W)$ for some valuation overring $W\neq K$. The first step is showing that $\Zar(D)$ is equal to $\comp(X)$ for some finite $X$.
\begin{prop}\label{prop:fincomp}
Let $D$ be a local integral domain. If $\Zarmin(D)$ is Noetherian, then there are valuation overrings $W_1,\ldots,W_n$ of $D$, $W_i\neq K$, such that $\Zar(D)=\comp(W_1)\cup\cdots\cup\comp(W_n)$.
\end{prop}
\begin{proof}
Let $R:=\Kr(D)$ be the Kronecker function ring of $D$. Then, the extension $N:=MR$ of the maximal ideal $M$ of $D$ is a proper ideal of $R$, and the prime ideals containing $N$ correspond to the valuation overrings of $R$ where $N$ survives, i.e., to the valuation overrings of $D$ centered on $M$.

Since $\Zarmin(D)$ is Noetherian, so is $\Max(R)$; since $R$ is treed (being a B\'ezout domain), by Lemma \ref{lemma:treed-MaxR} $N$ has only finitely many minimal primes. Thus, there are finitely many valuation overrings of $D$, say $W_1,\ldots,W_n$, such that every $V\in\Zarmin(D)$ is contained in one $W_i$. We claim that $\Zar(D)=\comp(W_1)\cup\cdots\cup\comp(W_n)$. Indeed, let $V$ be a valuation overring of $D$. Since $\Zar(D)$ is compact, $V$ contains some minimal valuation overring $V'$, and by construction $V'\in\comp(W_i)$ for some $i$; in particular, $W_i\supseteq V'$. The valuation overrings containing $V'$ (i.e., the valuation overrings of $V'$) are linearly ordered; thus, $V$ must be comparable with $W_i$, i.e., $V\in\comp(W_i)$. The claim is proved.
\end{proof}

The following result can be seen as a generalization of the classical fact that, if $X=\{V_1,\ldots,V_n\}$ is finite, then $\Zar(A(X))$ is the union of the various $\Zar(V_i)$ (since $A(X)$ will be a Pr\"ufer domain and its localization at the maximal ideals will be a subset of $X$).
\begin{prop}\label{prop:compX}
Let $D$ be an integral domain and let $X\subseteq\Zar(D)$ be a finite set. Then, $\Zar(A(\comp(X)))=\comp(X)$.
\end{prop}
\begin{proof}
Since $\comp(V)\subseteq\comp(W)$ if $V\subseteq W$, we can suppose without loss of generality that the elements of $X$ are pairwise incomparable. Let $X=\{V_1,\ldots,V_n\}$, $A_i:=A(\comp(V_i))$ and let $A:=A(\comp(X))=A_1\cap\cdots\cap A_n$. Note that $D\subseteq A$, and thus the quotient field of $A$ coincides with the quotient field of $D$ and of the $V_i$.

If $V\in\comp(X)$, then clearly $A\subseteq V$; thus, $\comp(X)\subseteq\Zar(A)$.

Conversely, let $V\in\Zar(A)$, and let $\mathfrak{m}_i$ be the maximal ideal of $V_i$. Then, $\mathfrak{m}_i\subseteq W$ for every $W\in\comp(V_i)$; in particular, $\mathfrak{m}_i\subseteq A_i$. Therefore, $P:=\mathfrak{m}_1\cap\cdots\cap\mathfrak{m}_n\subseteq A$; since $A\subseteq V$, this implies that $PV\subseteq V$.

Suppose $V\notin\comp(X)$, and let $T:=V\cap V_1\cap\cdots\cap V_n$. Since the rings $V,V_1,\ldots,V_n$ are pairwise incomparable, $T$ is a B\'ezout domain whose localizations at the maximal ideals are $V,V_1,\ldots,V_n$. In particular, $V$ is flat over $T$, and each $\mathfrak{m}_i$ is a $T$-module; hence,
\begin{equation*}
PV=\left(\bigcap_{i=1}^n\mathfrak{m}_i\right)V=\bigcap_{i=1}^n\mathfrak{m}_iV.
\end{equation*}
Since $V$ is not comparable with $V_i$, for each $i$, the set $\mathfrak{m}_i$ is not contained in $V$; in particular, the family $\{\mathfrak{m}_iV\mid i=1,\ldots,n\}$ is a family of $V$-modules not contained in $V$. Since the $V$-submodules of the quotient field $K$ are linearly ordered, the family has a minimum, and thus $\bigcap_{i=1}^n\mathfrak{m}_iV$ is not contained in $V$. However, this contradicts $PV\subseteq V$; hence, $V$ must be in $\comp(X)$, and $\Zar(A)=\comp(X)$.
\end{proof}

The proof of part \ref{prop:locAi:loc} of the following proposition closely follows the proof of \cite[Proposition 1.19]{heinzer_noethintersect-II}.
\begin{prop}\label{prop:locAi}
Let $X:=\{V_1,\ldots,V_n\}$ be a finite family of valuation overrings of the domain $D$, and suppose that $V_iV_j=K$ for every $i\neq j$, where $K$ is the quotient field of $D$. Let $A_i:=A(\comp(V_i))$, and let $A:=A(\comp(X))$. Then:
\begin{enumerate}[(a)]
\item\label{prop:locAi:loc} each $A_i$ is a localization of $A$;
\item\label{prop:locAi:exist} for each ideal $I$ of $A$, there is an $i$ such that $IA_i\neq A_i$;
\item\label{prop:locAi:AiAj} if $i\neq j$, then $A_iA_j=K$.
\end{enumerate}
\end{prop}
\begin{proof}
\ref{prop:locAi:loc} By induction and symmetry, it is enough to prove that $B:=A_2\cap\cdots\cap A_n$ is a localization of $A$. Let $J$ be the Jacobson radical of $B$: then, $J\neq(0)$, since it contains the intersection $\mathfrak{m}_{V_2}\cap\cdots\cap\mathfrak{m}_{V_n}$. Furthermore, if $W\neq K$ is a valuation overring of $V_1$, then $J\nsubseteq W$, since otherwise (as in the proof of Proposition \ref{prop:compX})  $\mathfrak{m}_{V_2}\cap\cdots\cap\mathfrak{m}_{V_n}$ would be contained in $\mathfrak{m}_W\cap(W\cap V_2\cap\cdots\cap V_n)$, against the fact that $\{W,V_2,\ldots,V_n\}$ are independent valuation overrings.

Hence, for every such $W$ we can apply \cite[Proposition 1.13]{heinzer_noethintersect-II} to $D:=B\cap W$, obtaining that $B$ is a localization of $D$, say $B=S^{-1}D$, where $S$ is a multiplicatively closed subset of $D$; in particular, there is a $s_W\in S\cap\mathfrak{m}_W$. Each $s_W$ is in $B\cap A_1=A$ (since $\mathfrak{m}_W$ is contained in every member of $\comp(V_1)$); let $T$ be the set of all $s_W$. Then,
\begin{equation*}
T^{-1}A=T^{-1}(B\cap A_2)=T^{-1}B\cap T^{-1}A_1.
\end{equation*}
Each $s_W$ is a unit of $B$, and thus $T^{-1}B=B$. On the other hand, no valuation overring $W\neq K$ of $V_1$ can be an overring of $T^{-1}A_1$, since $T$ contains $s_W$, which is inside the maximal ideal of $W$. Since $\Zar(A_1)=\comp(V_1)$, it follows that $T^{-1}A_1=K$, and thus $T^{-1}A=B$; in particular, $B$ is a localization of $A$.

\ref{prop:locAi:exist} Without loss of generality, we can suppose $I=P$ to be prime. There is a valuation overring $W$ of $A$ whose center on $A$ is $P$; since $\Zar(A)=\comp(X)$ by Proposition \ref{prop:compX}, there is a $V_i$ such that $W\in\comp(V_i)$. Hence, $PA_i\neq A_i$.

\ref{prop:locAi:AiAj} By Proposition \ref{prop:compX}, $\Zar(A_i)\cap\Zar(A_j)=\{K\}$. It follows that $K$ is the only common valuation overring of $A_iA_j$; in particular, $A_iA_j$ must be $K$.
\end{proof}

By \cite[Proposition 4.3]{starloc}, Proposition \ref{prop:locAi} can also be rephrased by saying that the set $\{A_1,\ldots,A_n\}$ is a \emph{Jaffard family} of $A$, in the sense of \cite[Section 6.3]{fontana_factoring}.

\begin{prop}\label{prop:MaxAcomp}
Let $D$ be an integrally closed domain; suppose that $\Zar(D)=\comp(V_1)\cup\cdots\cup\comp(V_n)$, where $X:=\{V_1,\ldots,V_n\}$ is a family of incomparable valuation overrings of $D$ such that $V_iV_j=K$ if $i\neq j$. Then:
\begin{enumerate}[(a)]
\item\label{prop:MaxAcomp:inj} the restriction of the center map $\gamma$ to $X$ is injective;
\item\label{prop:MaxAcomp:max} $|\Max(D)|\geq|X|$.
\end{enumerate}
\end{prop}
\begin{proof}
\ref{prop:MaxAcomp:inj} If $P$ is the image of both $V_i$ and $V_j$, then $P$ survives in both $A_i$ and $A_j$: however, since $A_i$ and $A_j$ are localizations of $A$ (Proposition \ref{prop:locAi}\ref{prop:locAi:loc}), $A_P$ would be a common overring of $A_i$ and $A_j$, against the fact that $A_iA_j=K$ (Proposition \ref{prop:locAi}\ref{prop:locAi:AiAj}). Therefore, the center map is injective on $X$.

\ref{prop:MaxAcomp:max} Let $M$ be a maximal ideal: then, there is a unique $i$ such that $MA_i\neq A_i$. In particular, $M$ can contain only one element of $\gamma(X)$, namely $\gamma(V_i)$; thus, $|\Max(D)|\geq|\gamma(X)|=|X|$, as claimed.
\end{proof}

We are ready to prove the pivotal result of the paper.
\begin{teor}\label{teor:pvd}
Let $D$ be an integrally closed local domain. If $\Zarmin(D)$ is a Noetherian space, then $D$ is a pseudo-valuation domain.
\end{teor}
\begin{proof}
Since $D$ is local, by Proposition \ref{prop:fincomp} there are $W_1,\ldots,W_n$, not equal to $K$, such that $\Zar(D)=\comp(W_1)\cup\cdots\cup\comp(W_n)$. By eventually passing to bigger valuation domains, we can suppose without loss of generality that $W_iW_j=K$ if $i\neq j$; since $D$ is local, by Proposition \ref{prop:MaxAcomp}\ref{prop:MaxAcomp:max} we have $1\geq n$, and so $\Zar(D)=\comp(V)$ for some $V\neq K$.

Let $\Delta$ be the set of $W\in\Zar(D)$ such that $\comp(W)=\Zar(D)$; then, $\Delta$ is a chain, and thus it has a minimum in $\Zar(D)$, say $V_0$ (explicitly, $V_0$ is the intersection of the elements of $\Delta$); furthermore, clearly $V_0\in\Delta$. Since $V\in\Delta$, we have $V_0\subseteq V$, and in particular $V_0\neq K$. Let $M$ be the maximal ideal of $V_0$: then, $M$ is contained in every $W\in\comp(V_0)=\Zar(D)$, and thus $M\subseteq D$.

Consider now the diagram
\begin{equation*}
\begin{tikzcd}
D\arrow[hookrightarrow]{d}\arrow[two heads]{r}{\pi} & D/M\arrow[hookrightarrow]{d}\\
V_0\arrow[two heads]{r }{\pi} & V_0/M.
\end{tikzcd}
\end{equation*}
Clearly, $D=\pi^{-1}(D/M)$; let $F_1$ be the quotient field of $D/M$. By Lemma \ref{lemma:pullback-Zar}, the set of minimal valuation overrings of $D$ is homeomorphic to $\Zarmin(V_0/M|D/M)$, which thus is Noetherian; by Proposition \ref{prop:noncampo}, it follows that either $D/M$ is a field and $\trdeg_{D/M}(V_0/M)=1$ (in which case $D$ is a pseudo-valuation domain with associated valuation domain $V_0$) or $\trdeg_{F_1}(V_0/M)=0$.

In the latter case, we note that $D/M$ is integrally closed in $V_0/M$, since $D/M$ is the intersection of all the elements of $\Zar(V_0/M|D/M)$; hence, $V_0/M$ is the quotient field of $D/M$. If $D/M$ is not a field, by the same argument of the first part of the proof it follows that $\Zar(D/M)=\comp(W_0)$ for some valuation overring $W_0\neq F_1$; however, this contradicts the choice of $V_0$, because $\pi^{-1}(W_0)$ would be comparable with every element of $\Zar(D)$. Hence, it must be $V_0/M=D/M$, i.e., $V_0=D$; that is, $D$ is a valuation domain and, in particular, a pseudo-valuation domain.
\end{proof}

With this result, we can find the possible structures of $\Zar(D)$ and $\Zarmin(D)$, when $D$ is local and $\Zarmin(D)$ is Noetherian. Indeed, $D$ is a pseudo-valuation domain; let $V$ be its associated valuation overring. Then, we have two cases: either $D=V$ (i.e., $D$ itself is a valuation domain) or $D\neq V$.

In the first case, $\Zarmin(D)$ is a singleton, while $\Zar(D)$ is homeomorphic to $\Spec(D)$; in particular, $\Zar(D)$ is linearly ordered, and it is a Noetherian space if and only if $\Spec(D)$ is Noetherian.

In the second case, we can separate $\Zar(D)$ into two parts: $\Zarmin(D)$ and $\Delta:=\Zar(D)\setminus\Zarmin(D)$. The former must be isomorphic to $\Zarmin(L|F)=\Zar(L|F)\setminus\{L\}$ (where $F$ and $L$ are the residue fields of $D$ and $V$, respectively); on the other hand, the latter is linearly ordered, and is composed by the valuation overrings of $V$, so in particular it is homeomorphic to $\Spec(V)$, which is (set-theoretically) equal to $\Spec(D)$. In other words, $\Zar(D)$ is composed by a long ``stalk'' ($\Delta$), under which there is an infinite family of minimal valuation overrings. In particular, we get the following.

\begin{prop}\label{prop:Zar-pvd-noeth}
Let $D,V,F,L$ as above. Then:
\begin{enumerate}[(a)]
\item $\Zarmin(D)$ is Noetherian if and only if $\Zar(L|F)$ is Noetherian.
\item $\Zar(D)$ is Noetherian if and only if $\Zar(L|F)$ and $\Spec(V)$ are Noetherian.
\end{enumerate}
\end{prop}
\begin{proof}
If $\Zarmin(D)$ is Noetherian, then $\Zarmin(L|F)$ is Noetherian as well. By Propositions \ref{prop:trdeg02} and \ref{prop:trdeg1}, $\Zar(L|F)$ is Noetherian.

If $\Zar(D)$ is Noetherian, so are $\Spec(D)=\Spec(V)$ and $\Delta\simeq\Zar(L|F)$ (in the notation above). Conversely, if $\Zar(L|F)$ and $\Spec(V)$ are Noetherian then so are $\Zarmin(D)$ and $\Delta$, and thus also $\Zarmin(D)\cup\Delta=\Zar(D)$ is Noetherian.
\end{proof}

Furthermore, we can now apply Propositions \ref{prop:trdeg02} and \ref{prop:trdeg1} to characterize when $\Zar(L|F)$ is Noetherian (see the following Corollary \ref{cor:global}).

\medskip

We now study the non-local case.
\begin{lemma}\label{lemma:locpvd}
Let $D$ be an integral domain such that $D_M$ is a PVD for every $M\in\Max(D)$ and, for every $M$, let $V(M)$ be the valuation overring associated to $D_M$. Then, the space $\{V(M)\mid M\in\Max(D)\}$ is homeomorphic to $\Max(D)$.
\end{lemma}
\begin{proof}
Let $\Delta:=\{V(M)\mid M\in\Max(D)\}$. If $\gamma$ is the center map, then $\gamma(V(M))=M$ for every $M$; thus, $\gamma$ restricts to a bijection between $\Delta$ and $\Max(D)$. Since $\gamma$ is continuous and closed, it follows that it is a homeomorphism.
\end{proof}

\begin{teor}\label{teor:global}
Let $D$ be an integrally closed domain. Then:
\begin{enumerate}[(a)]
\item\label{teor:global:Zarmin} $\Zarmin(D)$ is Noetherian if and only if $\Max(D)$ is Noetherian and $\Zarmin(D_M)$ is Noetherian for every $M\in\Max(D)$;
\item\label{teor:global:Zar} $\Zar(D)$ is Noetherian if and only if $\Spec(D)$ is Noetherian and $\Zar(D_M)$ is Noetherian for every $M\in\Max(D)$.
\end{enumerate}
\end{teor}
\begin{proof}
\ref{teor:global:Zarmin} If $\Zarmin(D)$ is Noetherian, then $\Max(D)$ is Noetherian since it is the image of $\Zarmin(D)$ under the center map, while each $\Zarmin(D_M)$ is Noetherian since they are subspaces of $\Zarmin(D)$.

Conversely, suppose that $\Max(D)$ is Noetherian and that $\Zar(D_M)$ is Noetherian for every $M\in\Max(D)$. By the latter property and Theorem \ref{teor:pvd}, every $D_M$ is a PVD; by Lemma \ref{lemma:locpvd}, the space $\Delta:=\{V(M)\mid M\in\Max(D)\}$ (in the notation of the lemma) is homeomorphic to $\Max(D)$, and thus Noetherian. Let $\beta$ be the map sending a $W\in\Zarmin(D)$ to $V(\mathfrak{m}_W\cap D)$.

Let $X$ be any subset of $\Zarmin(D)$, and let $\Omega$ be an open cover of $X$; without loss of generality, we can suppose $\Omega=\{\B(f_\alpha)\mid\alpha\in A\}$, where the $f_\alpha$ are elements of $K$. Then, $\Omega$ is also a cover of $X':=\{\beta(V)\mid V\in X\}$; since $X'$ is compact (being a subset of the Noetherian space $\Delta$), there is a finite subfamily of $\Omega$, say $\Omega':=\{\B(f_1),\ldots,\B(f_n)\}$, that covers $X'$. For each $i$, let $X_i:=\{V\in X\mid f_i\in\beta(V)\}$; then, $X=X_1\cup\cdots\cup X_n$. We want to find, for each $i$, a finite subset $\Omega_i\subset\Omega$ that is a cover of $X_i$.

Fix thus an $i$, let $f:=f_i$, and let $I:=(D:_Df)$ be the conductor ideal. For every $M\in\Max(D)$, let $Z(M):=\gamma^{-1}(M)\cap X_i=\{V\in X_i\mid \mathfrak{m}_V\cap D=M\}$, where $\gamma$ is the center map. The union of the $Z(M)$ is $X_i$; we separate the cases $I\nsubseteq M$ and $I\subseteq M$.

If $I\nsubseteq M$, then $1\in ID_M=(D_M:_{D_M}f)$, and thus $f\in D_M$; hence, in this cases $\B(f)$ contains $Z(M)$. 

Suppose $I\subseteq M$; clearly, we can suppose $Z(M)\neq\emptyset$. We claim that in this case $M$ is minimal over $I$. Indeed, if there is a $V\in Z(M)$ then $f\in V$, and thus $f\in\beta(V)$; therefore, $f\in D_P$ for every prime ideal $P\subsetneq M$ (since $D_P\supsetneq\beta(V)$ for every such $P$), and thus $I\nsubseteq P$. Therefore, $M$ is minimal over $I$. By Lemma \ref{lemma:treed-MaxR}, $I$ has only finitely many minimal primes; hence, there are only finitely many $M$ such that $I\subseteq M$ and $Z(M)\neq\emptyset$. For each of these $M$, the set of valuation domains in $X$ centered on $M$ is a subset of $\Zarmin(D_M)$, and thus it is compact; hence, for each of them, $\Omega$ admits a finite subcover $\Omega(M)$. It follows that $\Omega_i:=\{\B(f)\}\cup\bigcup\Omega(M)$ is a finite subset of $\Omega$ that is a cover of $X_i$.

Hence, $\bigcup_i\Omega_i$ is a finite subset of $\Omega$ that covers $X$; thus, $X$ is compact. Since $X$ was arbitrary, $\Zarmin(D)$ is Noetherian.

\ref{teor:global:Zar} If $\Zar(D)$ is Noetherian, then $\Spec(D)$ and every $\Zar(D_M)$ are Noetherian.

Conversely, suppose that $\Spec(D)$ is Noetherian and that $\Zar(D_M)$ is Noetherian for every $M\in\Max(D)$. By the previous point, $\Zarmin(D)$ is Noetherian. Furthermore, if $P\in\Spec(D)\setminus\Max(D)$ then $D_P$ is a valuation domain; hence, $\Zar(D)\setminus\Zarmin(D)$ is homeomorphic to $\Spec(D)\setminus\Max(D)$, which is Noetherian by hypothesis. Being the union of two Noetherian subspaces, $\Zar(D)$ itself is Noetherian.
\end{proof}

\begin{cor}\label{cor:global}
Let $D$ be an integral domain that is not a field, and let $L$ be a field containing $D$; suppose that $D$ is integrally closed in $L$. Then, $\Zar(L|D)$ (respectively $\Zarmin(L|D)$) is Noetherian if and only if the following hold:
\begin{itemize}
\item $L$ is the quotient field of $D$;
\item $\Spec(D)$ is Noetherian (resp., $\Max(D)$ is Noetherian);
\item for every $M\in\Max(D)$, the ring $D_M$ is a pseudo-valuation domain such that $\Zar(L|F)$ is Noetherian, where $F$ is the residue field of $D_M$ and $L$ is the residue field of the associated valuation overring of $D_M$.
\end{itemize}
\end{cor}
\begin{proof}
Join Proposition \ref{prop:noncampo}, Theorem \ref{teor:global} and Proposition \ref{prop:Zar-pvd-noeth}.
\end{proof}

For our last result, we recall that the \emph{valuative dimension} $\dim_v(D)$ of an integral domain $D$ is the supremum of the dimensions of the valuation overrings of $D$; a domain $D$ is called a \emph{Jaffard domain} if $\dim(D)=\dim_v(D)<\infty$, while it is a \emph{locally Jaffard domain} if $D_P$ is a Jaffard domain for every $P\in\Spec(D)$ \cite{jaffarddomains}. Any locally Jaffard domain is Jaffard, while the converse does not hold \cite[Example 3.2]{jaffarddomains}. The class of Jaffard domains includes, for example, (finite-dimensional) Noetherian domain, Pr\"ufer domains and universally catenarian domains.
\begin{prop}\label{prop:dimv}
Let $D$ be an integrally closed integral domain of finite dimension, and suppose that $\Zarmin(D)$ is a Noetherian space. Then:
\begin{enumerate}[(a)]
\item\label{prop:dimv:dimv} $\dim_v(D)\in\{\dim(D),\dim(D)+1\}$;
\item\label{prop:dimv:jaff} $D$ is locally Jaffard if and only if $D$ is a Pr\"ufer domain.
\end{enumerate}
\end{prop}
\begin{proof}
\ref{prop:dimv:dimv} Let $M$ be a maximal ideal of $D$. Then, $\Zarmin(D_M)$ is Noetherian, and thus $D_M$ is a pseudo-valuation domain; by \cite[Proposition 2.9]{jaffarddomains}, $\dim_v(D_M)=\dim(D_M)+\trdeg_FL$, where $F$ is the residue field of $D_M$ and $L$ is the residue field of the associated valuation ring of $D_M$. By Propositions \ref{prop:Zar-pvd-noeth} and \ref{prop:trdeg02}, $\trdeg_FL\leq 1$, and thus $\dim_v(D_M)\leq\dim(D_M)+1$. Hence, $\dim_v(D)\leq\dim(D)+1$; since $\dim_v(D)\geq\dim(D)$ always, we have the claim.

\ref{prop:dimv:jaff} If $D$ is a Pr\"ufer domain then it is locally Jaffard. Conversely, if $D$ is locally Jaffard, then $\dim_v(D_P)=\dim(D_P)$ for every prime ideal $P$ of $D$. Take any maximal ideal $M$, and let $F,L$ as above; using $\dim_v(D_M)=\dim(D_M)+\trdeg_FL$, it follows that $\trdeg_FL=0$. Since $D$ (and so $D_M$) is integrally closed, it must be $F=L$, i.e., $D_M$ itself is a valuation domain. Therefore, $D$ is a Pr\"ufer domain.
\end{proof}
Note that there are domains $D$ that are Jaffard domains and have $\Zar(D)$ Noetherian, but are not Pr\"ufer domains. Indeed, the construction presented in \cite[Example 3.2]{jaffarddomains} gives a ring $R$ with two maximal ideals, $M$ and $N$, such that $R_M$ is a two-dimensional valuation ring while $R_N$ is a one-dimensional pseudo-valuation domain with $\dim_v(R_N)=2$; in particular, it is a Jaffard domain that is not Pr\"ufer. Choosing $k=K(Z_1)$ in the construction (or, more generally, choosing $k$ such that $K(Z_1,Z_2)$ is finite over $k$), the Zariski space of $R_N$ is Noetherian (being homeomorphic to $\Zar(K(Z_1,Z_2)|k)$, which is Noetherian by Proposition \ref{prop:trdeg1}), and thus $\Zar(R)$ is Noetherian.


\begin{thebibliography}{10}

\bibitem{jaffarddomains}
David~F. Anderson, Alain Bouvier, David~E. Dobbs, Marco Fontana, and Salah
  Kabbaj.
\newblock On {J}affard domains.
\newblock {\em Exposition. Math.}, 6(2):145--175, 1988.

\bibitem{atiyah}
M.~F. Atiyah and I.~G. Macdonald.
\newblock {\em Introduction to {C}ommutative {A}lgebra}.
\newblock Addison-Wesley Publishing Co., Reading, Mass.-London-Don Mills, Ont.,
  1969.

\bibitem{bourbaki_ac}
Nicolas Bourbaki.
\newblock {\em Commutative {A}lgebra. {C}hapters 1--7}.
\newblock Elements of Mathematics (Berlin). Springer-Verlag, Berlin, 1989.
\newblock Translated from the French, Reprint of the 1972 edition.

\bibitem{dobbs_fedder_fontana}
David~E. Dobbs, Richard Fedder, and Marco Fontana.
\newblock Abstract {R}iemann surfaces of integral domains and spectral spaces.
\newblock {\em Ann. Mat. Pura Appl. (4)}, 148:101--115, 1987.

\bibitem{fontana_krr-abRs}
David~E. Dobbs and Marco Fontana.
\newblock Kronecker function rings and abstract {R}iemann surfaces.
\newblock {\em J. Algebra}, 99(1):263--274, 1986.

\bibitem{fifolo_transactions}
Carmelo~A. Finocchiaro, Marco Fontana, and K.~Alan Loper.
\newblock The constructible topology on spaces of valuation domains.
\newblock {\em Trans. Amer. Math. Soc.}, 365(12):6199--6216, 2013.

\bibitem{compact-intersections}
Carmelo~A. Finocchiaro and Dario Spirito.
\newblock Topology, intersections and flat modules.
\newblock {\em Proc. Amer. Math. Soc.}, 144(10):4125--4133, 2016.

\bibitem{topologically-defined}
Marco Fontana.
\newblock Topologically defined classes of commutative rings.
\newblock {\em Ann. Mat. Pura Appl. (4)}, 123:331--355, 1980.

\bibitem{fontana_factoring}
Marco Fontana, Evan Houston, and Thomas Lucas.
\newblock {\em Factoring {I}deals in {I}ntegral {D}omains}, volume~14 of {\em
  Lecture Notes of the Unione Matematica Italiana}.
\newblock Springer, Heidelberg; UMI, Bologna, 2013.

\bibitem{fontana_loper}
Marco Fontana and K.~Alan Loper.
\newblock An historical overview of {K}ronecker function rings, {N}agata rings,
  and related star and semistar operations.
\newblock In {\em Multiplicative ideal theory in commutative algebra}, pages
  169--187. Springer, New York, 2006.

\bibitem{gilmer}
Robert Gilmer.
\newblock {\em Multiplicative {I}deal {T}heory}.
\newblock Marcel Dekker Inc., New York, 1972.
\newblock Pure and Applied Mathematics, No. 12.

\bibitem{pvd}
John~R. Hedstrom and Evan~G. Houston.
\newblock Pseudo-valuation domains.
\newblock {\em Pacific J. Math.}, 75(1):137--147, 1978.

\bibitem{heinzer_noethintersect-II}
William Heinzer.
\newblock Noetherian intersections of integral domains. {II}.
\newblock 311:107--119, 1973.

\bibitem{hochster_spectral}
Melvin Hochster.
\newblock Prime ideal structure in commutative rings.
\newblock {\em Trans. Amer. Math. Soc.}, 142:43--60, 1969.

\bibitem{hub-kneb}
Roland Huber and Manfred Knebusch.
\newblock On valuation spectra.
\newblock In {\em Recent advances in real algebraic geometry and quadratic
  forms ({B}erkeley, {CA}, 1990/1991; {S}an {F}rancisco, {CA}, 1991)}, volume
  155 of {\em Contemp. Math.}, pages 167--206. Amer. Math. Soc., Providence,
  RI, 1994.

\bibitem{olberding_noetherianspaces}
Bruce Olberding.
\newblock Noetherian spaces of integrally closed rings with an application to
  intersections of valuation rings.
\newblock {\em Comm. Algebra}, 38(9):3318--3332, 2010.

\bibitem{olberding_affineschemes}
Bruce Olberding.
\newblock Affine schemes and topological closures in the {Z}ariski-{R}iemann
  space of valuation rings.
\newblock {\em J. Pure Appl. Algebra}, 219(5):1720--1741, 2015.

\bibitem{olberding_topasp}
Bruce Olberding.
\newblock Topological aspects of irredundant intersections of ideals and
  valuation rings.
\newblock In {\em Multiplicative Ideal Theory and Factorization Theory:
  Commutative and Non-Commutative Perspectives}. Springer Verlag, 2016.

\bibitem{ribenboim}
Paulo Ribenboim.
\newblock {\em The {T}heory of {C}lassical {V}aluations}.
\newblock Springer Monographs in Mathematics. Springer-Verlag, New York, 1999.

\bibitem{roquette-holomorphy}
Peter Roquette.
\newblock Principal ideal theorems for holomorphy rings in fields.
\newblock {\em J. Reine Angew. Math.}, 262/263:361--374, 1973.
\newblock Collection of articles dedicated to Helmut Hasse on his seventy-fifth
  birthday.

\bibitem{schwartz-compactification}
Niels Schwartz.
\newblock Compactification of varieties.
\newblock {\em Ark. Mat.}, 28(2):333--370, 1990.

\bibitem{ZarNoeth}
Dario Spirito.
\newblock Non-compact subsets of the {Z}ariski space of an integral domain.
\newblock {\em Illinois J. Math.}, 60(3-4):791--809, 2016.

\bibitem{starloc}
Dario Spirito.
\newblock Jaffard families and localizations of star operations.
\newblock {\em J. Commut. Algebra}, to appear.

\bibitem{zariski_sing}
Oscar Zariski.
\newblock The reduction of the singularities of an algebraic surface.
\newblock {\em Ann. of Math. (2)}, 40:639--689, 1939.

\bibitem{zariski_comp}
Oscar Zariski.
\newblock The compactness of the {R}iemann manifold of an abstract field of
  algebraic functions.
\newblock {\em Bull. Amer. Math. Soc.}, 50:683--691, 1944.

\bibitem{zariski_samuel_II}
Oscar Zariski and Pierre Samuel.
\newblock {\em Commutative {A}lgebra. {V}ol. {II}}.
\newblock Springer-Verlag, New York, 1975.
\newblock Reprint of the 1960 edition, Graduate Texts in Mathematics, Vol. 29.

\end{thebibliography}
\end{document}